\documentclass[11pt,leqno]{amsart}
\usepackage[mathcal]{eucal}
\usepackage{hyperref}
\usepackage{amssymb,amsbsy}

\newtheorem{theorem}{Theorem}[section]
\newtheorem{lemma}[theorem]{Lemma}
\newtheorem{proposition}[theorem]{Proposition}
\newtheorem{corollary}[theorem]{Corollary}

\theoremstyle{definition}
\newtheorem{example}[theorem]{Example}
\newtheorem{fact}[theorem]{Fact}

\newtheorem*{gtheorem*}{Gasch\"utz' Lemma}
\newtheorem*{nproblem*}{Neumanns' Problem}

\numberwithin{equation}{section}

\DeclareMathOperator{\PSL}{PSL}
\DeclareMathOperator{\Aut}{Aut}
\DeclareMathOperator{\Alt}{Alt}
\DeclareMathOperator{\Sym}{Sym}
\DeclareMathOperator{\SL}{SL}
\DeclareMathOperator{\GL}{GL}
\DeclareMathOperator{\tr}{tr}

\title[Generating pairs that fail to lift]{Generating pairs of
  projective special linear groups that fail to lift}

\author[J. Boschheidgen]{Jan Boschheidgen} 
\address{J.\ Boschheidgen: Departamento de Matematicas, Universidad Aut\'{o}noma de Madrid, and Instituto de Ciencias Matem\'{a}ticas, 28049 Madrid, Spain} \email{jan.boschheidgen@uam.es}

\author[B. Klopsch]{Benjamin Klopsch} 
\address{B.\ Klopsch:
  Mathematisches Institut, Heinrich-Heine-Universit\"at, 40225
  D\"usseldorf, Germany} \email{klopsch@math.uni-duesseldorf.de}

\author[A. Thillaisundaram]{Anitha Thillaisundaram} 
 \thanks{The third author acknowledges the support from the Alexander
  von Humboldt Foundation and from the Forscher-Alumni-Programm of the Heinrich-Heine-Universit\"at D\"usseldorf (HHU); she thanks HHU for its hospitality.}
\address{A.\ Thillaisundaram: School of Mathematics and Physics,
  University of Lincoln, LN6 7TS Lincoln, United Kingdom}
\email{anitha.t@cantab.net}

\dedicatory{In memory of Wolfgang Gasch\"utz (1920--2016) on the
  centenary of his birth}



\begin{document}

\begin{abstract}
  The following problem was originally posed by B.\,H.~Neumann and
  H.~Neumann.  Suppose that a group $G$ can be generated by $n$
  elements and that $H$ is a homomorphic image of~$G$.  Does there
  exist, for every generating $n$-tuple $(h_1,\ldots, h_n)$ of $H$, a
  homomorphism $\vartheta \colon G \to H$ and a generating $n$-tuple
  $(g_1,\ldots,g_n)$ of $G$ such that
  $(g_1^{\, \vartheta},\ldots,g_n^{\, \vartheta}) = (h_1,\ldots,h_n)$?
  
  M.\,J.~Dunwoody gave a negative answer to this question, by means of
  a carefully engineered construction of an explicit pair of soluble
  groups.  Via a new approach we produce, for $n = 2$, infinitely many
  pairs of groups $(G,H)$ that are negative examples to the Neumanns'
  problem.  These new examples are easily described: $G$ is a free
  product of two suitable finite cyclic groups, such as
  $C_2 \ast C_3$, and $H$ is a suitable finite projective special
  linear group, such as $\PSL(2,p)$ for a prime $p \ge 5$.  A small
  modification yields the first negative examples $(G,H)$ with $H$
  infinite.
\end{abstract}

\maketitle


\section{Introduction}

The following question about lifting finite generating tuples along
homomorphisms of groups was originally raised by B.\,H.~Neumann and
H.~Neumann~\cite{NeNe51}.

\begin{nproblem*} Suppose that $G$ is a group which can be generated
  by $n$ elements and let $H$ be a homomorphic image of $G$. Does
  there exist, for every generating $n$-tuple $(h_1,\ldots, h_n)$ of
  $H$, a homomorphism $\vartheta \colon G \to H$ and a generating
  $n$-tuple $(g_1,\ldots,g_n)$ of $G$ such that
  $(g_1^{\, \vartheta},\ldots,g_n^{\, \vartheta}) = (h_1,\ldots,h_n)$?
\end{nproblem*}

By means of an ingenious counting trick, W.~Gasch\"utz showed
in~\cite{Ga56} that under an additional finiteness condition the
answer is yes; indeed, a significantly stronger result holds.

\begin{gtheorem*} Let $G$ be a group which can be generated by $n$
  elements and let $\vartheta \colon G \to H$ be a surjective
  homomorphism with \emph{finite} kernel.  Then, for every generating
  $n$-tuple $(h_1,\ldots,h_n)$ of $H$, there exists a generating
  $n$-tuple $(g_1,\ldots, g_n)$ of $G$ such that
  $(g_1^{\, \vartheta},\ldots,g_n^{\, \vartheta})=(h_1,\ldots,h_n)$.
\end{gtheorem*}

A profinite version of Gasch\"utz' Lemma has manifold applications in
the theory of profinite groups; compare~\cite[\S 2]{Lu01}.  For a
recent generalisation to metrisable compact groups see~\cite{CoGe18}.

However, in general one cannot expect that arbitrary generating tuples
will lift along a single, fixed epimorphism.  For instance, take
$\vartheta \colon C_{\infty}\to C_5$, a homomorphism from an infinite
cyclic group onto a group of order~$5$.  Then, of course, only two of
the four generators of $C_5$ lift along $\vartheta$ to generators of
$C_\infty$.  As Gasch\"utz pointed out, this simple example
illustrates the inherent limitation in his theorem, but falls short of
settling the Neumanns' Problem (in the negative). 

Eventually M.\,J.~Dunwoody~\cite[\S 3]{Du63} gave a negative answer to
the Neumanns' Problem (and indeed negative answers to several related
questions), by means of a carefully engineered construction of an
explicit pair of $2$-generated groups $(G,H)$, where $H$ is a
homomorphic image of $G$, but $H$ admits a generating pair that does
not lift back to a generating pair of $G$ along any homomorphism.  In
this example, $H$ is a (nilpotent of class $2$)-by-(nilpotent of class
$2$) group of order $5^3 \, 11^{10}$; the group $G$ is an extension of an
(infinite) abelian group by~$H$.

It is largely unknown how `frequently' negative examples to the
Neumanns' Problem should be expected to occur, when one restricts $G$
or $H$ to more special classes of groups.  The main purpose of the
present paper is to provide, for $n=2$, an infinite family of negative
examples $(G,H)$ to the Neumanns' Problem whose nature is rather
different from the groups arising from Dunwoody's construction.
Indeed, our main examples arise from free products of cyclic groups
$G \cong C_2 \ast C_m$ mapping onto finite projective special linear
groups $H \cong \PSL(2,q)$.

\begin{theorem}\label{thm:main-result}
  For the following pairs of $2$-generated groups, the group $H$ is a
  homomorphic image of $G$, but there exist generating pairs
  $(h_1,h_2)$ of $H$ such that for all homomorphisms
  $\vartheta \colon G \to H$ and all generating pairs $(g_1,g_2)$ of
  $G$, we have
  $(g_1^{\, \vartheta},g_2^{\, \vartheta}) \ne (h_1,h_2)$.

  \begin{enumerate}
  \item[(i)] $G = C_2 \ast C_3$ and $H = \PSL(2,q)$, where
    $q$ is a prime power such that $q \ge 4$, but $q \ne 9$;
  \item[(ii)] $G = C_2 \ast C_p$ and $H = \PSL(2,q)$, where $q = p^k$
    is a power of a prime $p \ge 3$ such that $q \ge 7$, but
    $q \ne 9$;
  \item[(iii)] $G = C_2 \ast C_m$ and $H = \PSL(2,q)$, where
    $m \in \mathbb{N}$ and $q = p^k$ is a power of a prime $p$ such
    that $q \equiv_4 3$, but $q \ne 3$,  and
    \[
    p \mid m \quad \vee \quad \gcd(m,(q + 1)/2) \ge 3 \quad \vee \quad
    \gcd(m,(q - 1)/2) \ge 3;
    \]
  \item[(iv)] $G = C_3 \ast C_3$ and $H =\PSL(2,q)$, where
    $q$ is a prime power such that $q\ge 5$.
  \end{enumerate}
\end{theorem}

The groups of the form $C_2 \ast C_m$, where $m \ge 3$, arise in
number theory as Hecke groups, acting on the upper half of the complex
plane; e.g., see~\cite{IkSaCa09}.  The most prominent member of this
family is the modular group $\PSL(2,\mathbb{Z}) \cong C_2 \ast C_3$,
which clearly maps onto $\PSL(2,p)$ for any prime~$p$.  As expected,
the infinite dihedral group $G \cong C_2 \ast C_2$ does \emph{not} yield
negative examples to the Neumanns' Problem; see
Proposition~\ref{Diedergruppe}.

As a first step toward Theorem~\ref{thm:main-result}, one needs to
verify that the groups $H$ are actually homomorphic images of the
relevant free product~$G$.  Fortunately, the subgroups of finite
projective special linear groups of degree $2$ were completely
described by L.\,E.\ Dickson; compare~\cite[Thm.~II.8.27]{Hu67}.
Based on this classification, A.\,M.~Macbeath~\cite{Mc67} studied
triples $(A,B,C)$ of $\SL(2,q)$ such that $ABC=1$ with a view toward
determining $\langle \overline{A},\overline{B} \rangle \le \PSL(2,q)$
according to the traces of $A,B,C$.  In particular, he showed that
$\PSL(2,q)$ is $(2,3)$-generated unless $q = 9$; we refer to
\cite{Pe17} for a complete overview about $(2,3)$-generation of
arbitrary (projective) special linear groups over finite fields and
references to related results.

In~\cite{LaRo89}, U.\ Langer and G.\ Rosenberger use Macbeath's
results to determine necessary and sufficient conditions for
$\PSL(2,q)$ to be the quotient of a given triangle group, with
torsion-free kernel.  Generalising a result of J.\,L.\ Brenner and J.\
Wiegold, they establish (ibid., Satz~5.3) the following theorem.  For
every prime power $q \not \in \{ 2,4,5,9 \}$ and all integers $m \ge 2$,
$n \ge 3$, the group $\PSL(2,q)$ is $(m,n)$-generated provided that
$m, n$ are possible element orders in $\PSL(2,q)$.

We recall that $m$ is a possible element order in $\PSL(2,q)$, where
$q = p^k$ is a power of a prime~$p$, if and only if
\[
m = p \qquad \vee \qquad
  m \mid \frac{q + 1}{\gcd(2,q-1)} \qquad \vee \qquad
  m \mid \frac{q - 1}{\gcd(2,q-1)}. 
\]
This explains, for instance, the assumptions in
Theorem~\ref{thm:main-result}(iii), following the condition
$q \equiv_4 3$.  Generation properties of $\PSL(2,q)$ for
$q \in \{ 2,4,5,9 \}$ can, of course, be dealt with by direct
computation: $\PSL(2,9) \cong \Alt(6)$ is not $(2,3)$-generated, but
it is $(3,3)$-generated; $\PSL(2,5) \cong \Alt(5)$ is
$(2,5)$-generated, but generating pairs lift, as we state in
Proposition~\ref{pro:PSL25-lift}.

In Dunwoody's example and all our examples thus far, the group $G$ is
$2$-generated and infinite, whereas the homomorphic image $H$ is
finite.  By a small modification, we obtain the first negative
examples $(G,H)$ to the Neumanns' Problem with $H$ infinite.

\begin{corollary} \label{cor:infinite-H}
  There are negative examples $(G,H)$ to the Neumanns' Problem, where
  $H$ is infinite.  For instance, $G = \PSL(2,\mathbb{Z})$ and
  $H = \PSL(2,5) \times G/G''$ yield such an example: $H$ is a
  homomorphic image of $G$, but $H$ admits generating pairs that do
  not lift to generating pairs for $G$ along any homomorphism.
\end{corollary}

Similar to Dunwoody, we focus in this paper exclusively on
$2$-generated groups and on lifting generating pairs.  It remains a
challenge to produce negative examples $(G,H)$ to the Neumanns'
Problem for $n\ge 3$, where $G$ and $H$ are $n$-generated and certain
generating $n$-tuples do not lift back.  It is conceivable that such
examples can be found as a byproduct of a systematic study of the
Neumanns' Problem for pairs $(G,H)$, where $G$ is a finite free
product of cyclic groups.  Theorem~\ref{thm:main-result}(iv) can be
seen as a first minor step toward such a systematic study.

Finally, we think it would be interesting to investigate to what
degree the extra condition $q \equiv_4 3$ in
Theorem~\ref{thm:main-result}(iii) is needed.  Our proof relies on
this assumption and breaks down, for instance, for $G = C_2 \ast C_m$
and $H = \PSL(2,q)$ with
$(m,q) \in \{ (7,13), (19,37) , (20,41), (21,41) \}$: one can check by
direct computer calculation that, in these cases, the trace invariants
associated to relevant generating pairs $(\overline{A},\overline{B})$
assume all possible values, apart from~$2$.  This stands in contrast
to the situations that we deal with.

\medskip



\noindent \textbf{Acknowledgements.}  The initial ideas for our paper
stem from an unpublished manuscript of the second author, which
contained a self-contained number-theoretic proof that, for every
prime $p \ge 5$, there exist generating pairs of $H = \PSL(2,p)$ which
do not lift to $G = \PSL(2,\mathbb{Z}) \cong C_2 \ast C_3$.  The third
author thanks Sandro Mattarei and Alex Bartel for helpful discussions.
We acknowledge the referee's feedback which led to improvements of the
presentation.


\section{Preliminary set-up}\label{sec2}
Let $G$ be a $2$-generated group.  There is a natural semi-regular
action of $\Aut(G)$ on the set
\[
\Gamma_G = \{ (g_1,g_2) \in G \times G \mid \langle g_1, g_2 \rangle
= G \}
\]
of all generating pairs of $G$ under which $\alpha \in \Aut(G)$ sends
$(g_1,g_2) \in \Gamma_G$ to $(g_1^{\, \alpha},g_2^{\, \alpha})$.

We are more interested in the action of the automorphism group
$\mathcal{A} = \Aut(F)$ of a free group $F = \langle x,y \rangle$ of
rank two on $\Gamma_G$ defined as follows. There is a bijection
\begin{align*}
  \Phi \colon \Gamma_G & \to \mathcal{E} = \{ \eta \mid \eta \colon F
                         \twoheadrightarrow G \text{ a surjective homomorphism} \}, \\
  (g_1,g_2) & \mapsto \quad \eta_{g_1,g_2} \colon F \twoheadrightarrow
              G, \, w(x,y) \mapsto w(g_1,g_2),
\end{align*}
where $w(x,y)$ denotes any word in $x,y$.  The group $\mathcal{A}$
permutes the elements of $\mathcal{E}$, via
$\eta^\alpha = \alpha^{-1} \eta$ for $\eta \in \mathcal{E}$ and
$\alpha \in \mathcal{A}$, and this action pulls back along $\Phi$ to an
action of $\mathcal{A}$ on $\Gamma_G$.

It is well-known that $\mathcal{A}$ is generated by the basic
Nielsen transformations
\begin{align*}
  \alpha_1 \colon F \to F, & \quad w(x,y) \mapsto w(x^{-1},y), \\
  \alpha_2 \colon F \to F, & \quad w(x,y) \mapsto w(xy,y), \\
  \alpha_3 \colon F \to F, & \quad w(x,y) \mapsto w(y,x).
\end{align*}
Application of these to an element $(g_1,g_2) \in \Gamma_G$ via $\Phi$ yields
\begin{equation}\label{nielsen}
  (g_1,g_2)^{\alpha_1} = (g_1^{-1},g_2), \quad (g_1,g_2)^{\alpha_2} =
  (g_1 g_2^{-1},g_2), \quad (g_1,g_2)^{\alpha_3} = (g_2,g_1).
\end{equation}

We note that, if $\vartheta \colon G \twoheadrightarrow H$ is an
epimorphism of groups, then $\vartheta$ induces a map
$\vartheta^* \colon \Gamma_G \to \Gamma_H, (g_1,g_2) \mapsto (g_1^{\,
  \vartheta},g_2^{\, \vartheta})$
which commutes with the action of $\mathcal{A}$ on $\Gamma_G$ and
$\Gamma_H$ respectively; i.e., for every $\alpha \in \mathcal{A}$ and
all $(g_1,g_2) \in \Gamma_G$ we have
\begin{equation}\label{diagram}
  ((g_1,g_2)^\alpha) \vartheta^* = ((g_1,g_2) \vartheta^*)^\alpha.
\end{equation}

In general, neither $\mathcal{A}$ nor the permutation group $P$
induced by $\mathcal{A}$ and $\Aut(G)$ is necessarily going to act
transitively on $\Gamma_G$. 

\begin{example} \label{exa:Alt5}
  For instance, if $G = \PSL(2,5) \cong \Alt(5)$ is the alternating
  group of degree~$5$, then $\Gamma_G$ consists of $2280$
  elements.  They fall into nineteen $\Aut(G)$-orbits, each of length
  $120$, and into two $P$-orbits of length $1080 = 9 \cdot 120$ and
  $1200 = 10 \cdot 120$ respectively.  For details see \cite{Ha36} and
  \cite[\S 10]{NeNe51}.  By means of a direct calculation one can see
  that in the same example $\Gamma_G$ splits into three
  $\mathcal{A}$-orbits of size $600 = 5 \cdot 120$,
  $600 = 5 \cdot 120$, and $1080 = 9 \cdot 120$, respectively.
\end{example}

But under special circumstances $\mathcal{A}$ acts transitively on
$\Gamma_G$, e.g., if $G$ is abelian or if $G$ is free. We require the
following not so obvious fact.

\begin{fact}\label{one-orbit}
  If $G \cong C_m \ast C_n$ for $m,n\in \mathbb{N}$, then $\Gamma_G$
  consists of a single $\mathcal{A}$-orbit.
\end{fact}

This is a consequence of the Grushko--Neumann Theorem about free
groups, which can be proved, for instance, by cancellation arguments
due to R.\,C.~Lyndon; see \cite[Section~6.3]{Ro82} and
\cite[Appendix]{Co72}.  From Fact~\ref{one-orbit} and \eqref{diagram}
we obtain

\begin{corollary}\label{cor:orbits}
  Let $G = C_m \ast C_n$ for $m,n\in \mathbb{N}$. Suppose that $H$ is
  a homomorphic image of $G$, and let $(h_1,h_2)\in \Gamma_H$.  Then
  there exists $\vartheta \colon G \to H$ and $(g_1,g_2)\in \Gamma_G$
  with $(g_1^{\, \vartheta},g_2^{\, \vartheta})=(h_1,h_2)$ if and only if the
  $\mathcal{A}$-orbit of $(h_1,h_2)$ in $\Gamma_H$ contains an element
  $(\tilde{h}_1,\tilde{h}_2)$ such that
  ${\tilde{h}_1}^{\,m}=\tilde{h}_2^{\,n}=1$.
\end{corollary}

As above, let $H$ be a homomorphic image of~$G$.  We say that a subset
$\Omega\subseteq \Gamma_H$ is \emph{$(m,n)$-free} if there exists no
$(h_1,h_2)\in \Omega$ such that $h_1^{\, m} = h_2^{\, n} = 1$.
Furthermore, we refer to $(h_1,h_2)\in \Gamma_H$ as an
\emph{$(m,n)$-generating pair} if $h_1^{\, m} = h_2^{\, n} = 1$, i.e.,
if the orders of $h_1$ and $h_2$ divide $m$ and $n$, respectively.

We make use of a fruitful observation by D.\,G.~Higman;
see~\cite{Ne56}.  Let $\Omega \subseteq \Gamma_H$ be an
$\mathcal{A}$-orbit, and let
\[
c(\Omega)=\{[h_1,h_2]\mid (h_1,h_2)\in\Omega\}.
\]
Since the identities
\[
[x,y]=x^{-1}[x^{-1},y]^{-1}x=y^{-1}[xy^{-1},y]y=[y,x]^{-1}
\]
hold in any group, the induced Nielsen transformations on $\Gamma_G$
show that for any choice of $(h_1,h_2)\in \Omega$, we have
\[
c(\Omega)\subseteq \{[h_1,h_2]^g\mid g\in H\}\cup \{[h_2,h_1]^g \mid
g\in H\}.
\]
For instance, this shows that the order of $[h_1,h_2]$ is constant for
$(h_1,h_2)\in \Omega$.

We take interest in the special situation, where $H = \PSL(2,q)$ is a
projective special linear group, for a prime power~$q$.  Every element
$h\in H$ has matrix representatives $R(h),-R(h) \in \SL(2,q)$.  Let
$\tr \colon\SL(2,q) \to \mathbb{F}_q$ denote the usual trace map.  The
group commutator yields a canonical map from
$H \times H = \PSL(2,q) \times \PSL(2,q)$ to $\SL(2,p^k)$: for
$(h_1,h_2) \in H \times H$, the element
\[ 
[\![ h_1, h_2 ]\!] = [\pm R(h_1), \pm R(h_2)] \in \SL(2,q)
\]
is independent of the particular choice of representatives.  Combining
this map with the usual trace map, we obtain the \emph{trace
  invariant} 
\[
\tau(h_1,h_2) = \tr([\![h_1,h_2]\!])
\]
 which is constant on
$\mathcal{A}$-orbits $\Omega \subseteq \Gamma_H$.

As mentioned in the introduction, the (maximal) subgroups of finite
projective special linear groups are fully understood; therefore it is
possible to locate generating pairs satisfying additional properties.
Using this approach, D.~McCullough and M.~Wanderley~\cite{McWa11}
showed for instance that, for $q \ge 13$, every non-trivial element of
$\PSL(2,q)$ is a commutator of a generating pair.  For convenience we
state the following consequence of their work.

\begin{theorem}[{\cite[Thm.~2.1]{McWa11}\label{thm:trace-theorem}}]
  Let $H = \PSL(2,q)$, where $q$ is a prime power.  The set 
  $\mathcal{T}(H) = \{ \tau(h_1, h_2) \mid (h_1,h_2) \in \Gamma_H \}
  \subseteq \mathbb{F}_q$
  of trace invariants of generating pairs for $\PSL(2,q)$ satisfies
  \[
  \mathcal{T}(H) = 
  \begin{cases}
    \mathbb{F}_q \smallsetminus \{2\} & \text{if $q \in \{2, 4, 8\}$
      or
      $q \ge 13$,} \\
    \mathbb{F}_q \smallsetminus \{ 1,2 \} & \text{if $q \in \{
      3,9,11\}$,} \\
    \mathbb{F}_q \smallsetminus \{0,2,4\} = \{1,3\} & \text {if $q=5$,} \\
    \mathbb{F}_q \smallsetminus \{0,1,2\} = \{3,4,5,6\} & \text{if $q=7$.}
  \end{cases}
  \]
\end{theorem}

Next, we record a basic and well-known fact.

\begin{lemma}\label{order_p}
  Let $A \in \SL(2,q) \smallsetminus \{ I, -I \}$, where $q = p^k$ is
  a power of a prime~$p \ge 3$. Then
  \begin{enumerate}
  \item[(a)] $\tr(A) = 2$ if and only if $A$ has order $p$, and
  \item[(b)] $\tr(A)= -2$ if and only if $A$ has order $2p$.
  \end{enumerate}
\end{lemma}

\begin{proposition} \label{pro:key} Let $q$ be a prime power such that
  $q \equiv_4 3$.  Suppose that $\SL(2,q) = \langle A,B \rangle$,
  where $A^4 = I$. Then $\tr([A,B]) \ne -2$.
\end{proposition}

\begin{proof}
  For a contradiction, assume that $\tr([A,B]) = -2$.  Conjugating $A$
  and $B$ by a suitable element of $\GL(2,q)$, we may suppose that
  \[ 
   [A,B] =
  \begin{pmatrix}
    -1 & x \\
    0 & -1
  \end{pmatrix}, \qquad \text{where $x \in \mathbb{F}_q$.}
  \]
  Since $\PSL(2,q) = \langle \overline{A}, \overline{B} \rangle$ is
  not abelian, we deduce that $x \ne 0$.  Since $A$ has order $4$ in
  $\SL(2,q)$, it follows that $\tr(A)=0$, hence we may write
  \[
  A = \begin{pmatrix}
    a & b\\
    c & -a
  \end{pmatrix}
  \qquad \text{for suitable $a,b,c \in \mathbb{F}_q$.}
  \]
  Since
  \[
  B^{-1}AB = A [A,B] = 
  \begin{pmatrix}
    a & b\\
    c & -a
  \end{pmatrix}
  \begin{pmatrix}
    -1 & x \\
    0 & -1
  \end{pmatrix}
  =
  \begin{pmatrix}
    -a & x a-b\\
    -c & x c+a
  \end{pmatrix}
  \]
  also has trace~$0$, we deduce that $c=0$.  From $\det A=1$ we
  conclude that $a^2 = -1$, in contradiction to~$q \equiv_4 3$.
\end{proof}

Using Theorem~\ref{thm:trace-theorem}, we deduce from
Proposition~\ref{pro:key} the following key corollary.

\begin{corollary}\label{cor:general-argument}
  Let $q$ be a prime power such that $q \equiv_4 3$, but $q \ne 3$,
  and let $m \in \mathbb{N}$.  Suppose that $H = \PSL(2,q)$ is
  $(2,m)$-generated.  Then $G = C_2 \ast C_m$ and $H$ yield a negative
  example $(G,H)$ to the Neumanns' Problem.
\end{corollary}

\begin{proof}
  By Theorem~\ref{thm:trace-theorem}, we find a generating pair
  $(h_1,h_2) \in \Gamma_H$, where $h_1 = \overline{A}$ and
  $h_2 = \overline{B}$ for $A, B \in \SL(2,q)$, such that
  $\tr([A,B]) = \tau(h_1,h_2) = -2$.  By Corollary~\ref{cor:orbits},
  it suffices to show that $(h_1,h_2)$ represents an
  $\mathcal{A}$-orbit that is $(2,m)$-free.

  Suppose that $(h_1', h_2') = (\overline{A'},\overline{B'})$, for
  $A', B' \in \SL(2,q)$, is a $(2,m)$-generating pair for $H$.  Then
  $(A',B')$ is a generating pair for $\SL(2,q)$ and $A'^4 = I$.  By
  Proposition~\ref{pro:key}, we have $\tr([A',B']) \ne -2$.  Hence
  $\tau(h_1',h_2') \ne \tau(h_1,h_2)$, and $(h_1',h_2')$ cannot lie
  in the $\mathcal{A}$-orbit of $(h_1,h_2)$.
\end{proof}


\section{Proof of the main theorem and corollary}

We continue to use the notation introduced in Section~\ref{sec2}.

\begin{proof}[Proof of Theorem~\ref{thm:main-result}(iii)]
  This follows directly from Corollary~\ref{cor:general-argument}.
\end{proof}

Our proof of Theorem~\ref{thm:main-result}(i) uses group-theoretic
results, due to G.\,A.~Miller.

\begin{proof}[Proof of Theorem~\ref{thm:main-result}(i)]
  Consider the group $H = \PSL(2,q)$, where $q \ge 4$, but
  $q \ne 9$.  By Corollary~\ref{cor:orbits},
  it suffices to show that there is an $\mathcal{A}$-orbit in
  $\Gamma_H$ that is $(2,3)$-free.

  First suppose that $p \ne 2$ and $q \not \in \{5, 7\}$.  Using
  Theorem~\ref {thm:trace-theorem}, we find a generating pair
  $(h_1,h_2) \in \Gamma_H$ such that $\tau(h_1,h_2) = 0$.  This
  implies that $[h_1,h_2] \ne 1$ has order~$2$.  A classical result of
  G.\,A.~Miller~\cite[\S 1]{Miller} implies that any group generated
  by an element of order two and an element of order three whose
  commutator has order two is isomorphic to either $\Alt(4)$ or
  $\Alt(4) \times C_2$.  Thus the $\mathcal{A}$-orbit of $(h_1,h_2)$
  is $(2,3)$-free.

  Now suppose that $p = 2$ or $q \in \{5,7\}$.  Using Theorem~\ref
  {thm:trace-theorem}, we find a generating pair
  $(h_1,h_2) \in \Gamma_H$ such that $\tau(h_1,h_2) \in \{1,-1\}$.
  This implies that $[h_1,h_2] \ne 1$ has order~$3$.  By~\cite[\S
  3]{Miller}, any group generated by an element of order two and an
  element of order three whose commutator has order three is solvable.
  (Indeed, the finitely presented group
  $K = \langle x,y \mid x^2 = y^3 = [x,y]^3 = 1 \rangle$ has derived
  length~$3$; one has $\lvert K:K'' \rvert = 54$ and
  $K'' \cong C_\infty \times C_\infty$.)  Thus the $\mathcal{A}$-orbit
  of $(h_1,h_2)$ is $(2,3)$-free.
\end{proof}

For completeness we remark that it is easy to check that every
generating pair of $\PSL(2,2) \cong \Sym(3)$ and
$\PSL(2,3) \cong \Alt(4)$ respectively lifts to a generating pair of
$C_2 \ast C_3 \cong \PSL(2,\mathbb{Z})$ along a suitable epimorphism.

\begin{proof}[Proof of Theorem~\ref{thm:main-result}(ii)]
  Consider the group $H = \PSL(2,q)$, where $q = p^k$ is a power of a
  prime $p \ge 3$ such that $q \ge 7$, but $q \ne 9$.  By
  Corollary~\ref{cor:orbits}, it suffices to show that there is an
  $\mathcal{A}$-orbit of $\Gamma_H$ that is $(2,p)$-free.  If
  $q \equiv_4 3$, we apply Corollary~\ref{cor:general-argument}, as in
  the proof of part~(iii).

  Now suppose that $q \equiv_4 1$.  The argument we provide actually
  works for all prime powers $q \ge 11$.  Suppose that
  $\PSL(2,q) = \langle h_1,h_2\rangle $ with
  $h_1^{\, 2} = h_2^{\, p} = 1$, where $h_1 = \overline{A}$ and
  $h_2 = \overline{B}$ for $A,B \in \SL(2,q)$.  Then
  $\langle A,B \rangle = \SL(2,q)$, and since there is only one
  conjugacy class of elements of order $4$ in $\SL(2,q)$, we may
  suppose that
  \[
  A = 
  \begin{pmatrix}
    0 & 1\\
    -1 & 0
  \end{pmatrix}.
  \]
  We write 
  \[
  B = 
  \begin{pmatrix}
    a & b\\
    c & d
  \end{pmatrix}, \qquad \text{where $a,b,c,d \in \mathbb{F}_p$.}
  \]
  Since $B$ has order $p$ or $2p$, Lemma~\ref{order_p} shows that the
  trace of $\tr(B) \in \{ 2, -2 \}$.  Put $s = b-c\in \mathbb{F}_q$.  We have
  \[
  [A,B] =\left[ 
    \begin{pmatrix}
      0 & 1\\
      -1 & 0
    \end{pmatrix},
    \begin{pmatrix}
      a & b\\
      c & d
    \end{pmatrix}
  \right] = 
  \begin{pmatrix}
    a^2+c^2 & ab+cd\\
    ab+cd & b^2+d^2
  \end{pmatrix}.
  \]
  Thus $(h_1,h_2) \in \Gamma_H$ has trace invariant
  \[
  \tr([A,B]) = a^2+b^2+c^2+d^2=(a+d)^2+(b-c)^2-2 = s^2+2.
  \]
  By~Theorem~\ref{thm:trace-theorem}, there is an $\mathcal{A}$-orbit
  in $\Gamma_H$ such that the associated trace invariant is not of the
  form $s^2+2$ for $s \in \mathbb{F}_q$.  Such an $\mathcal{A}$-orbit
  is $(2,p)$-free.
\end{proof}

In preparation of the proof of Theorem~\ref{thm:main-result}(iv), we
establish two lemmata that could also be checked directly, using a computer.

\begin{lemma}\label{lem:5}
  If $\SL(2,5) = \langle A, B \rangle$, where
  $A^3, B^3 \in \{ I, -I \}$,  then $\tr([A,B]) \ne -2$.
\end{lemma}

\begin{proof}
  By Lemma~\ref{order_p}, it suffices to show that $[A,B]$ does not
  have order~$10$.  For this it suffices to show that
  $[\overline{A},\overline{B}] \in \PSL(2,5)$ does not have order~$5$. 

  Recall that $\PSL(2,5) \cong \Alt(5)$.  Hence it is enough to prove
  that in $\Alt(5)$ no commutator of two $3$-cycles yields a
  $5$-cycle.  Using conjugation in $\Sym(5)$, it suffices to consider
  a single commutator:
  \[
    [(1\,2\,3) , (1\,4\,5)] = (3\,2\,1) (5\,4\,1) (1\,2\,3) (1\,4\,5)
    = (1\,4\,2). \qedhere
  \]
\end{proof}

\begin{lemma}\label{lem:7}
  If $\SL(2,7) = \langle A,B \rangle$, where
  $A^3, B^3 \in \{ I, -I\}$, then $\tr([A,B]) \not\in \{ 3, -3\}$.
\end{lemma}

\begin{proof}
  Elements $C \in \SL(2,7)$ with $\tr(C) \in \{ 3, -3 \}$ have
  order~$8$, and the image of such an element in $\PSL(2,7)$ has
  order~$4$.  Hence it suffices to show that
  $[\overline{A},\overline{B}] \in \PSL(2,7)$ does not have order~$4$.
 
  Recall that $\PSL(2,7)$ acts $2$-transitively on the projective line
  $\mathbf{P}^1 (\mathbb{F}_7) = \mathbb{F}_7 \cup \{ \infty \}$.  The
  stabiliser of any pair of points is cyclic of order $3$, and the
  stabiliser of any $2$-element set is dihedral of order~$6$.  For
  instance, the stabiliser of $\{ 0, \infty \}$ is generated by
  \begin{align*}
    & t \colon x \mapsto -\tfrac{1}{x}, && \text{i.e.,} \quad t = (0\,\infty)
                                           (1\,6) (2\,3) (4\,5), \\
    & u_0 \colon x \mapsto  2x, && \text{i.e.,}\quad  u_0 = (0)
                                   (1\,2\,4) (3\,6\,5) (\infty).
  \end{align*}
  Elements of order $4$ act fixed-point-freely
  on~$\mathbf{P}^1 (\mathbb{F}_7)$.  The elements
  $\overline{A}, \overline{B}$ induce fractional linear
  transformations $u,v$ on $\mathbf{P}^1 (\mathbb{F}_7)$, with
  $\lvert \mathrm{Fix}(u) \rvert = \lvert \mathrm{Fix}(v) \rvert = 2$.
  Using conjugation and replacing $u$ by $u^{-1}$, if necessary, we
  may assume that $u = u_0$.

  For a contradiction, assume that $[u,v]$ has order~$4$ so that
  $\mathrm{Fix}([u,v]) = \varnothing$.  We conclude that
  $\mathrm{Fix}(u) = \{0,\infty\}$ and $\mathrm{Fix}(u).v$ are
  disjoint; likewise $\mathrm{Fix}(v)$ and $\mathrm{Fix}(v).u$ are
  disjoint.  The elements of order $3$ form a single conjugacy class
  in $\PSL(2,7)$, and we write $v = u^w$ for suitable
  $w \in \PSL(2,7)$.  Conjugating by $g \in \langle t,u \rangle$ and
  replacing $u$ by $u^{-1}$, if necessary, we may assume that $0.w=1$.
  This leaves three options for $\infty.w$, namely $3$, $6$ and $5$.
  Accordingly, $v$ is one of the transformations
  \begin{align*}
    & v_1 = u^{w_1}, \text{ where } w_1 \colon x \mapsto
      \tfrac{3x+4}{x+4}, && \text{i.e.,}\quad  v_1 = (1) (0\,4\,2)
                            (\infty\,5\,6) (3), \\ 
    & v_2 = u^{w_2}, \text{ where } w_2 \colon x \mapsto
      \tfrac{6x+3}{x+3}, && \text{i.e.,}\quad v_2 = (1) 
                            (4\,3\,\infty) (0\,2\,5) (6), \\
    & v_3  = u^{w_3}, \text{ where } w_3
      \colon x \mapsto  \tfrac{5x+2}{x+2}, && \text{i.e.,}\quad  v_3 = (1)
                                              (0\,3\,6) (2\,4\,\infty) (5)
  \end{align*}
  or one of their inverses.  Replacing $v$ by its inverse, if
  necessary, we may assume that $v$ is one of $v_1, v_2, v_3$.  Direct
  computations show that
  \begin{align*}
    [u,v_1] & = (0\,2)(1\,4)(3\,5)(6\,\infty), \\
    [u,v_2] & = (3)(0\,\infty\,6\,1\,4\,2\,5), \\
    [u,v_3] & = (0\,5)(1\,\infty)(2\,4)(3\,6). \qedhere
  \end{align*}
\end{proof}

\begin{proof}[Proof of Theorem~\ref{thm:main-result}(iv)]
  Consider the group $H = \PSL(2,q)$, where $q$ is a prime power such
  that $q \ge 5$.  By Corollary~\ref{cor:orbits}, it suffices to show
  that there is an $\mathcal{A}$-orbit of $\Gamma_H$ that is
  $(3,3)$-free.

  If $q=5$, Theorem~\ref{thm:trace-theorem} and Lemma~\ref{lem:5}
  imply that at least one $\mathcal{A}$-orbit in $\Gamma_H$ is
  $(3,3)$-free.  If $q = 7$, Theorem~\ref{thm:trace-theorem} and
  Lemma~\ref{lem:7} show that there are $(3,3)$-free
  $\mathcal{A}$-orbits in $\Gamma_H$.  Now suppose that
  $q \not \in \{5,7\}$.  We show below that, if
  $(h_1,h_2) \in \Gamma_H$ and $h_1^3=h_2^3=1$, then
  $\tau(h_1,h_2) \ne 0$.  Thus Theorem~\ref{thm:trace-theorem} yields
  that at least one $\mathcal{A}$-orbit in $\Gamma_H$ is $(3,3)$-free.

  Let $(h_1,h_2)\in \Gamma_H$ be such that $h_1^3=h_2^3=1$. For a
  contradiction, assume that $\tau(h_1,h_2) = 0$.  Then
  $[h_1,h_2] \ne 1$ would have order~$2$.  But a classical result of
  G.\,A.~Miller~\cite[\S 2]{Miller} implies that any group generated
  by two elements of order three whose commutator has order two is
  soluble, in contradiction to $H=\PSL(2,q)$ being non-soluble.
  (Indeed, the finitely presented group
  $K = \langle x,y \mid x^3 = y^3 = [x,y]^2 = 1\rangle$ has order
  $288$ and derived length~$3$; one has $K/K' \cong C_3 \times C_3$
  and $K'' \cong C_2$.)
\end{proof}

\begin{proof}[Proof of Corollary~\ref{cor:infinite-H}]
  Suppose that $H_0 = \langle y_1, y_2 \rangle \cong \PSL(2,q)$ is a
  homomorphic image of
  $G = \langle x_1,x_2 \rangle \cong C_m \ast C_n$ under
  $x_1 \mapsto y_1$ and $x_2 \mapsto y_2$, where $q$ is a prime power
  such that $q \ge 5$ and $m,n \in \mathbb{N}$.  Suppose further that
  $(h_1,h_2) \in \Gamma_{H_0}$ does not lift to a generating pair
  $(g_1,g_2) \in \Gamma_G$ under any homomorphism~$G \to H_0$.

  Observe that $G$ modulo its second derived subgroup $G'$, is
  infinite: indeed,
  $G/G' = \langle \overline{x_1}, \overline{x_2} \rangle \cong C_m
  \times C_n$
  and the Kuro{\v s} Subgroup Theorem implies
  $G'/G'' \cong C_\infty^{\, (m-1)(n-1)}$.  The group $H_0$ is
  perfect, hence $H_0'' = H_0$.  Therefore $H = H_0 \times G/G''$ is a
  homomorphic image of $G$ under $x_1 \mapsto y_1 \overline{x_1}$ and
  $x_2 \mapsto y_2 \overline{x_2}$.  Furthermore
  $(h_1 \overline{x_1}, h_2 \overline{x_2})$ is a generating pair of
  $H$ that does not lift to a generating pair $(g_1,g_2) \in \Gamma_G$
  under any homomorphism $\vartheta \colon G \to H$, because otherwise
  $(h_1,h_2)$ would lift to $(g_1,g_2)$ under the composition
  $\vartheta \pi_1 \colon G \to H_0$, where $\pi_1 \colon H \to H_0$
  is the projection onto the first factor.
\end{proof}


\section{Complements}

In this section we give two complementary results.  First we show that
Theorem~\ref{thm:main-result}(ii) does not extend to $q=5$.

\begin{proposition} \label{pro:PSL25-lift} Let $G=C_2\ast C_5$ and
  $H=\PSL(2,5)$. Then every generating pair of $H$ lifts to a
  generating pair of $G$ along some epimorphism.
\end{proposition}

\begin{proof}
  In view of Corollary~\ref{cor:orbits}, it suffices to show that
  every $\mathcal{A}$-orbit in $\Gamma_H$ contains a
  $(2,5)$-generating pair.  The number $\mathcal{A}$-orbits is known
  to be~$3$; compare Example~\ref{exa:Alt5}.

  Consider in $\SL(2,5)$ the elements
  \[
  A = 
  \begin{pmatrix} 0 & 1 \\ -1 & 0 \end{pmatrix},
  \quad
  B = 
  \begin{pmatrix} 0 & 3 \\ 3 & 0 \end{pmatrix},
  \quad
  C = 
  \begin{pmatrix} 1 & 1 \\ 0 & 1 \end{pmatrix},
  \quad
  D = C^2 =  
  \begin{pmatrix} 1 & 2 \\ 0 & 1 \end{pmatrix}.
  \]
  It is straightforward to check that $(\overline{A},\overline{C})$,
  $(\overline{A},\overline{D})$ and $(\overline{B},\overline{D})$ are
  $(2,5)$-generating pairs for~$H$.  We claim that they are
  representatives for the three distinct $\mathcal{A}$-orbits
  in~$\Gamma_H$.  Indeed,
  \[
  [A,C] = 
  \begin{pmatrix} 1 & 1 \\ 1 & 2 \end{pmatrix},
  \quad 
  [A,D] = 
  \begin{pmatrix} 1 & 2 \\ 2 & 0 \end{pmatrix},
 \quad
  [B,D] = 
  \begin{pmatrix} 1 & 2 \\ 3 & 2 \end{pmatrix}
  \]
  so that 
  \[
  \tau(\overline{A},\overline{C}) = \tr([A,C]) = 3, \quad
  \tau(\overline{A},\overline{D}) = \tr([A,D]) = 1, \quad
  \tau(\overline{B},\overline{D}) = \tr([B,D]) = 3.
  \]  
  Thus it remains to show that $(\overline{A},\overline{C})$ and
  $(\overline{B},\overline{D})$ lie in distinct $\mathcal{A}$-orbits.
  This follows from the fact that $[A,C]$ is not conjugate to any of
  the matrices $[B,D]$, $[D,B]$, $-[B,D]$, $-[D,B]$.  (The latter two
  can be ruled out by their trace, the former two are ruled out by
  direct computation.)
\end{proof}

Finally, we justify that the infinite dihedral group
$G = C_2 \ast C_2$ does not yield negative examples to
the Neumanns' Problem

\begin{proposition}\label{Diedergruppe}
  Let $G = C_2 \ast C_2$ and let $H = \langle h_1, h_2 \rangle$ be a
  homomorphic image of~$G$.  Then there exists a homomorphism
  $\vartheta \colon G \to H$ and a generating pair
  $(g_1, g_2) \in \Gamma_G$ such that
  $(g_1^{\, \vartheta}, g_2^{\, \vartheta}) = (h_1,h_2)$.
\end{proposition}

\begin{proof}
  The infinite dihedral group
  $G = \langle x, y \mid x^2 = y^2 = 1 \rangle$ is just infinite; its
  non-trivial proper homomorphic images are exactly the finite
  dihedral groups $D_{2m}$, for $m \in \mathbb{N}$.  Avoiding trivial
  and easy cases, we may suppose that $H = D_{2m}$ for $m \ge 3$.

  The group $D_{2m}$ consists of $m$ shifts and $m$ reflections.  At
  least one of $h_1$ and $h_2$ is a reflection.  Let's suppose $h_1$
  is a reflection.  If $h_2$ is also a reflection, then there is a
  homomorphism mapping $g_1 = x$ to $h_1$ and
  $g_2 = y$ to~$h_2$.  If $h_2$ is a shift (necessarily of order $m$),
  then there is a homomorphism mapping $g_1 = x$ to $h_1$ and
  $g_2 = xy$ to~$h_2$.   
\end{proof}
 

\end{document}